\documentclass{amsart}
\usepackage{amsmath}
\usepackage{amssymb}
\usepackage{hyperref}

\newtheorem{theorem}{Theorem}[section]
\newtheorem{lemma}[theorem]{Lemma}

\theoremstyle{definition}
\newtheorem{definition}[theorem]{Definition}
\newtheorem{example}[theorem]{Example}

\theoremstyle{remark}

\begin{document}

\title[A note on "On Green's relations, $2^{0}$-regularity and quasi-ideals in $\Gamma$-semigroups]{A note on "On Green's relations, $2^{0}$-regularity and quasi-ideals in $\Gamma$-semigroups, Acta Math. Sin., Engl. Ser., 29, No. 3, 609-624(2013)"}

\author{\textbf{Kostaq Hila \and Jani Dine}}
\address{Department of Mathematics \& Computer Science, Faculty of Natural
Sciences\\
University of Gjirokastra\\Albania} \email{kostaq\_hila@yahoo.com, jani\_dine@yahoo.com}


\subjclass[2000]{20M10, 20M12, 20M17}

\keywords{$\Gamma$-semigroup, $(m,n)^{0}$-strongly regular, $2^{0}$-strongly regular, left (right) strongly regular.}

\begin{abstract}
The aim of this note is to point out some inaccuracies in our paper \cite{HD} and to fix them. Some new notions are introduced and properties of them are investigated. 
\end{abstract}

\maketitle
\section{Introduction and preliminaries}

In our paper published recently \cite{HD}, it has been pointed out a serios gap in the proof of Lemma 2.13 and Lemma 2.14, which was based on the result obtained in Theorem 1 \cite{Y-S}. But the proof of this theorem turns out to be wrong. This affected some results obtained in the section 2 of our paper. In order to fix this, we have to modify the two lemmas and to provide the correct proof of them. To achieve this, we need to introduce some new notions and to prove some properties of them. Also, independently from us, in \cite{NK} the author pointed out this gap. 

Let we recall the definition of $\Gamma$-semigroups.

Let $M$ and $\Gamma$ be two non-empty sets. Any map from $M\times
\Gamma\times M\rightarrow M$ will be called a $\Gamma$-%
\textit{multiplication} in $M$ and denoted by $(\cdot)_{\Gamma}$. The
result of this multiplication for $a, b\in M$ and $\alpha\in \Gamma$ is
denoted by $a\alpha b$. A $\Gamma$-\textit{semigroup} $M$ is an ordered
pair $(M, (\cdot)_{\Gamma})$ where $M$ and $\Gamma$ are non-empty sets and $%
(\cdot)_{\Gamma}$ is a $\Gamma$-multiplication on $M$ which satisfies
the following property
\begin{center}
$\forall (a, b,c, \alpha, \beta)\in M^{3}\times \Gamma^{2}, (a\alpha b)\beta
c=a\alpha (b\beta c)$.
\end{center}

Other definitions and notions can be found in \cite{HD}. 

In the following, we introduce the notion of strongly regular, left (right) strongly regular $\Gamma$-semigroups and through some examples we show that the class of this kind of $\Gamma$-semigroups is not empty. 

\begin{definition} A $\Gamma$-semigroup $M$ is said to be left (right) strongly regular if for all $a\in M$ and $\gamma\in \Gamma$, there exists an element $u\in M$ such that:
\begin{center}
$a=u\gamma a\gamma a (a=a\gamma a\gamma u)$.
\end{center}
\end{definition}

\begin{definition} A $\Gamma$-semigroup $M$ is said to be strongly regular if it is both left strongly regular and right strongly regular, that is, if for all $a\in M$ and $\gamma\in \Gamma$, there exist elements $u_1, u_2\in M$ such that:
\begin{center}
$a=u_2\gamma a\gamma a=a\gamma a\gamma u_1$.
\end{center}
\end{definition}

Let we illustrate this notion by giving the following examples.
\begin{example} Let $Q$ be the set of rational numbers and $\Gamma=\{-1, 1\}$. If we define as $\Gamma$-multiplication $(\cdot)_\Gamma$ the usual product of rational numbers, then it is clear that $(Q, (\cdot)_\Gamma)$ is a $\Gamma$-semigroup. This $\Gamma$-semigroup with zero is a strongly regular $\Gamma$-semigroup. Indeed: for every nonzero rational number $a$ and for all $\gamma\in \Gamma$, if we take $u=\frac{1}{a\gamma^{2}}$, then we have
\begin{center}
$a=a\cdot \gamma\cdot a\cdot \gamma\cdot u=u\cdot \gamma\cdot a\cdot \gamma\cdot a$.
\end{center}
For $a=0$ we have: $0=0\cdot \gamma\cdot 0\cdot \gamma \cdot 1=1\cdot \gamma\cdot 0\cdot \gamma\cdot 0$.
\end{example}
\begin{example} Let $R$ be the set of real numbers and $\Gamma$ a subset of positive real numbers. If we define as $\Gamma$-multiplication $(\cdot)_\Gamma$ the usual product of real numbers, then it is clear that $(R, (\cdot)_\Gamma)$ is a $\Gamma$-semigroup and further, it is a strongly regular $\Gamma$-semigroup. Indeed: if $a=0$, then for all $\gamma\in \Gamma$ we have $0=0\cdot \gamma\cdot 0\cdot \gamma \cdot 1=1\cdot \gamma\cdot 0\cdot \gamma\cdot 0$. For $a\neq 0$, if we take $u=\frac{1}{a\gamma^{2}}$, then we have
\begin{center}
$a=a\cdot \gamma\cdot a\cdot \gamma\cdot u=u\cdot \gamma\cdot a\cdot \gamma\cdot a$.
\end{center}
\end{example}

Considering the examples above from a common point of view, we can construct a strongly regular $\Gamma$-semigroup getting started from an arbitrary group $G^{0}$ with zero. Let us take as $\Gamma$ a subset of $G$, i.e. all its elements are nonzero. If we define in $G^{0}$ as $\Gamma$-multiplication $(\cdot)_\Gamma$ the product of the group $G^{0}$, then it is clear that $(G^{0}, (\cdot)_\Gamma)$ is a $\Gamma$-semigroup. Further, it is a strongly regular $\Gamma$-semigroup. Indeed: for $a=0$, we have $0=0\cdot \gamma \cdot 0\cdot \gamma\cdot e=e\cdot \gamma\cdot 0\cdot \gamma\cdot 0$. If $a\neq 0$, we take $u=\gamma^{-1}a^{-1}\gamma^{-1}$, then we have
\begin{center}
$a=a\cdot \gamma\cdot a\cdot \gamma\cdot u=u\cdot \gamma\cdot a\cdot \gamma\cdot a$.
\end{center} 

\begin{example} Let $(S, \cdot)$ be an arbitrary semigroup which is commutative and regular. For every element $a\in S$, there exists $x\in S$ such that:
\begin{center}
$a=axa=a^{2}x=xa^{2}.$\hspace{1cm} (1)
\end{center}
Thus, the semigroup $S$ is left regular and right regular. Since $S$ is regular, then there exists at least one idempotent element $e$ in $S$, and so we have the maximal group having $e$ as identity. We denote this group by $\Gamma$. If we define in $S$ as $\Gamma$-multiplication $(\cdot)_\Gamma$ the multiplication of the semigroup $S$, then it is clear that $(S, (\cdot)_\Gamma)$ is a commutative $\Gamma$-semigroup. Further, it is a strongly regular $\Gamma$-semigroup. Indeed: for all $a\in S$ and for all $\gamma\in \Gamma$, if we take $u=\gamma^{-2}a$, according to (1), we have
\begin{center}
$a=a\cdot \gamma\cdot a\cdot \gamma\cdot u=u\cdot \gamma\cdot a\cdot \gamma\cdot a$.
\end{center} 
\end{example}

Taking into consideration the definition of left regular semigroup and right regular semigroup, we can give the following equivalent definition for strongly regular $\Gamma$-semigroups:
\begin{definition} A $\Gamma$-semigroup $(M, (\cdot)_\Gamma)$ is said to be strongly regular if for every $\gamma\in \Gamma$, the semigroup $(M,\gamma)$ (cf. \cite[pp. 611]{HD}) is left regular semigroup (the usual one) and right regular semigroup. 
\end{definition}

Let us prove now the following lemma. 
\begin{lemma} Let $M$ be a $\Gamma$-semigroup. The following are equivalent:
\begin{enumerate}
	\item $M$ is left strongly regular.
	\item $a\mathcal{L} (a\gamma a)$, $\forall a\in M, \forall \gamma\in \Gamma$.
\end{enumerate}
\end{lemma}
\begin{proof} $(1)\Rightarrow (2)$. Let $M$ be strongly left regular and $a\in M$. If $x\in (a)_l$, then we have $x=a$ or $x=t\alpha a$ for some $t\in M$ and $\alpha \in \Gamma$. Since $M$ is strongly regular, $a=y\gamma (a\gamma a)$ for some $y\in M$ and for all $\gamma\in \Gamma$. 

If $x=a$, then $x=a=y\gamma (a\gamma a)$ for some $y\in M$ and for all $\gamma\in \Gamma$.

If $x=t\alpha a$, then $x=t\alpha a=t\alpha (y\gamma (a\gamma a))=(t\alpha y)\gamma (a\gamma a)=z\gamma a\gamma a$ for some $z\in M$ and for all $\gamma\in \Gamma$.

Hence $x\in (a\gamma a)_l$ for all $\gamma\in \Gamma$, thus $(a)_l\subseteq (a\gamma a)_l, \forall \gamma\in \Gamma$ (*).

Let now $a\in M$ and $\gamma\in \Gamma$. If $x\in (a\gamma a)_l$, then $x=a\gamma a$ or $x=t\alpha (a\gamma a)$ for some $t\in M$. In any case, $x=z\gamma a$ for some $z\in M$, thus $x\in (a)_l, \forall \gamma\in \Gamma$ (**). By (*) and (**) we have $(a)_l=(a\gamma a)_l, \forall a\in M, \forall \gamma\in \Gamma$. Thus we have $a\mathcal{L} (a\gamma a)$, $\forall a\in M, \forall \gamma\in \Gamma$.

$(2)\Rightarrow (1)$. Let $a\in M$. Then for all $\gamma\in \Gamma$, we have $a\mathcal{L} (a\gamma a)$. Thus we have: $a\in (a)_l=(a\gamma a)_l\Rightarrow a=a\gamma a$ or $a=x\mu (a\gamma a)$ for some $x\in M, \mu\in \Gamma$. If $a=a\gamma a$, then $a=a\gamma (a\gamma a)$. If $a=x\mu(a\gamma a)$, then $a=x\mu (x\mu a\gamma a)\gamma a=(x\mu x\mu a)\gamma (a\gamma a)=y\gamma a\gamma a$. Hence $M$ is left strongly regular.
\end{proof}

In a similar way, we can show the following lemma.
 
\begin{lemma} Let $M$ be a $\Gamma$-semigroup. The following are equivalent:
\begin{enumerate}
	\item $M$ is strongly right regular.
	\item $a\mathcal{R} (a\gamma a)$, $\forall a\in M, \forall \gamma\in \Gamma$.
\end{enumerate}
\end{lemma}

\begin{theorem} Let $M$ be a left strongly regular $\Gamma$-semigroup. The following statements hold:
\begin{enumerate}
	\item Every $\mathcal{L} (\mathcal{R})$-class of $M$ is a left (right) simple sub-$\Gamma$-semigroup.
	\item The $\Gamma$-semigroup $M$ is a disjoint union of left (right) simple sub-$\Gamma$-semigroups.
\end{enumerate}
\end{theorem}
\begin{proof} We will make the proof only for the $\mathcal{L}$-classes and left simple sub-$\Gamma$-semigroups. The proof for the $\mathcal{R}$-classes and right simple sub-$\Gamma$-semigroups is similar.

(1). By Lemma 1.7, since $M$ is left strongly regular, for all $a\in M$ and for all $\gamma\in \Gamma$, we have $a\mathcal{L}(a\gamma a)$. Let $a\mathcal{L} b$, so $b\in L_a$. Since $\mathcal{L}$ is a right congruence, then $a\mathcal{L} b$ implies $a\gamma a\mathcal{L}b\gamma a$ for all $\gamma\in \Gamma$. Thus, since $a\mathcal{L} (a\gamma a)$ for all $\gamma\in \Gamma$ and $\mathcal{L}$ is an equivalence relation, we have $b\gamma a\mathcal{L} a$. So, $b\gamma a\in L_a$ for all $\gamma\in \Gamma$, and therefore $L_a$ is a sub-$\Gamma$-semigroup of $M$. Let us prove now that for all $a\in M$, the sub-$\Gamma$-semigroup $L_a$ is left simple. We have to prove that for all $b\in L_a$, there exist $c\in L_a$ and $\gamma\in \Gamma$ such that $b=c\gamma a$. Since $L_a$ is a sub-$\Gamma$-semigroup of $M$, for all $\gamma\in \Gamma$, $b\gamma a\in L_a$. That is, there exists $x\in M$ and $\gamma_1\in \Gamma$, such that $b=x\gamma_1 b\gamma a$. Denote $c=x\gamma_1 b$ and so we have $(c)_l\subseteq (b)_l$. Since $M$ is a left strongly regular $\Gamma$-semigroup, for the element $x\in M$ and $\gamma_1\in \Gamma$, there exists $y\in M$ such that $x=y\gamma_1 x\gamma_1 x$. Thus we have:
\begin{center}
$b=(y\gamma_1 x)\gamma_1 (x\gamma_1 b\gamma a)=(y\gamma_1 x)\gamma_1 b=y\gamma_1(x\gamma_1 b)=y\gamma_1 c$
\end{center}
 which shows that $(b)_l\subseteq (c)_l$. So, $(c)_l=(b)_l$. Therefore, $c\in L_b=L_a$ and $b=c\gamma a$.
 
(2). It is an immediate corollary of (1), since the $\mathcal{L}$-classes are disjoint and their union is equal to $M$.
\end{proof}

We know that for the plain semigroups, each of the above proposition (1), (2) implies the semigroup is left (right) regular. We have our doubts about the case of $\Gamma$-semigroups. Thus the following problem arises:

\textbf{Problem 1.} What can we say about the converse of the above theorem? If it is not true, can we find counterexamples showing that each of the proposition (1) or (2) doesn't imply the fact that $\Gamma$-semigroup $M$ is left (right) strongly regular?

Let $M$ be a strongly regular $\Gamma$-semigroup. By the Definition 1.2, $M$ is at the same time a left strongly regular and a right strongly regular. By the Theorem 1.9, $M$ is a disjoint union of left simple sub-$\Gamma$-semigroups and at the same time a disjoint union of right simple sub-$\Gamma$-semigroups. Since the $\Gamma$-semigroups, which are at the same time left simple and right simple, are $\Gamma$-groups, a natural question is posed: What can we say about the strongly regular $\Gamma$-semigroups, are they a union of $\Gamma$-groups? The positive answer of it is given by the following theorem:
\begin{theorem} Let $M$ be a $\Gamma$-semigroup. The following statements are equivalent:
\begin{enumerate}
	\item $M$ is a union of $\Gamma$-groups.
	\item $M$ is a strongly regular $\Gamma$-semigroup.
	\item Every $\mathcal{H}$-class of $M$ is a $\Gamma$-group.
	\item $M$ is a disjoint union of $\Gamma$-groups.
\end{enumerate}
\end{theorem}
\begin{proof} $(1)\Rightarrow (2)$. Let us assume that (1) holds. For every element $a\in M$, there exists a $\Gamma$-group $G$ of the union such that $a\in G$. Then, for every $\gamma\in \Gamma$, $a\gamma a\in G_\gamma=(G, \gamma)$ and therefore, there exist an element $u\in G$ such that $a=u\gamma a\gamma a$. Thus, $M$ is a right strongly regular $\Gamma$-semigroup. In similar way, it can be shown that $M$ is a left strongly regular $\Gamma$-semigroup. Hence, $M$ is a strongly regular $\Gamma$-semigroup.

$(2)\Rightarrow (3)$. Let us assume that $M$ is a strongly regular $\Gamma$-semigroup. Thus, for every $a\in M$ and $\gamma\in \Gamma$, there exist $u_1, u_2\in M$ such that $a=a\gamma a\gamma u_1=u_2\gamma a\gamma a$. This implies that $a\mathcal{H}(a\gamma a)$. By the Green's Theorem for $\Gamma$-semigroup \cite[Theorem 2.2]{HD}, the $\mathcal{H}$-class $H_a$ is a $\Gamma$-group. 

$(3)\Rightarrow (4)$. Let us assume that (3) holds. For every element $a\in M$, the $\mathcal{H}$-class $H_a$ is a $\Gamma$-group. Since $\mathcal{H}$ is an equivalence relation, we have
\begin{center}
$M=\bigcup\limits_{a\in M} H_a$ and $\forall (a,b)\in M^{2}, (H_a=H_b)\vee (H_a\cap H_b=\emptyset)$.
\end{center}

$(4)\Rightarrow (1)$. It is obvious.
\end{proof}

By the Theorem 1.10 it follows that every strongly regular $\Gamma$-semigroup $M$ is regular. Indeed: for every element $a\in M$, there exists a $\Gamma$-group $G$ such that $a\in G$ and so we have $a\in a\Gamma G\Gamma a\subseteq a\Gamma M\Gamma a$.

The converse is not true in general, as we know from the algebraic theory of semigroups, there are regular semigroups which are not union of groups and therefore, by Theorem 1.10, they are not strongly regular $\Gamma$-semigroups. These arguments justify to some extent the term strongly regular $\Gamma$-semigroup.

If $M$ is a strongly regular $\Gamma$-semigroup with zero, 0, then $\mathcal{H}$-class of the element zero contains only the zero element, while every other class of a nonzero representative, does not contain the zero element. But at the same time, it is a $\Gamma$-group with zero the element 0. Thus, every $\Gamma$-semigroup with zero is a union of $\Gamma$-groups with zero.

Let $M$ be a strongly regular $\Gamma$-semigroup without zero and $a$ be an arbitrary element of $M$. Since $M$ is a union of $\Gamma$-groups, there exists a $\Gamma$-group $G$ of this union containing $a$. Then for every $\gamma_1\in \Gamma$, we have $a, a\gamma_1 a\in G_{\gamma_2}$, where $\gamma_2$ is an arbitrary element of $\Gamma$. Therefore, there are $u, u_1, u_2\in M$ such that
\begin{center}
$a=u_1\gamma_1 a\gamma_2 a=a\gamma_2 a\gamma_1 u_2$.\hspace{2cm} (*) 
\end{center}

Conversely, it is obvious that, if for every $a\in M$ and for every $\gamma_1, \gamma_2\in \Gamma$, there exist $u_1, u_2\in M$ such that (*) holds, then $M$ is a strongly regular $\Gamma$-semigroup.

Thus, we may give the following new equivalent definition.
\begin{definition} A $\Gamma$-semigroup $M$ is said to be strongly regular if 
\begin{center}
$\forall (a, \gamma_1, \gamma_2)\in M\times \Gamma^{2}, \exists (u_1, u_2)\in M^{2}, a=u_1\gamma_1 a\gamma_2 a=a\gamma_2 a\gamma_1 u_2$.
\end{center}
\end{definition}

According to the Definition 1.6, a $\Gamma$-semigroup $M$ is strongly regular when for every $\gamma\in \Gamma$, the semigroup $M_\gamma=(M, \gamma)$ is left regular and right regular. A natural question is posed: The same as for $\Gamma$-groups, if for an element $\gamma\in \Gamma$, the semigroup $(M, \gamma)$ is left regular and right regular, is the semigroup $(M, \gamma)$ a left regular and right regular semigroup for every $\gamma\in \Gamma$? Thus the following problems arise:

\textbf{Problem 2.} If $M$ is a $\Gamma$-semigroup such that for an element $\gamma\in \Gamma$, the semigroup $(M, \gamma)$ is a left and a right regular semigroup, is $M$ a strongly regular $\Gamma$-semigroup? 

\textbf{Problem 3.} If $M$ is a regular and right (left) strongly regular $\Gamma$-semigroup, is it a strongly regular $\Gamma$-semigroup?

%%%%%%%%%%%%%%%%%%%%%%%%%%%%%%%%%%%%%%%%%%%%%%%%%%%%%%%%%%%%%%%%%%%%%%%%%%%%%%%%%%%%%%%%%%%%%%%%%%%%%%%%%%%%%%%%

\section {$2^{0}$-strongly regularity}

In this section, we introduce the definition of $(m ,n)^{0}$-strongly regularity in $\Gamma$-semigroups in general and that of $2^{0}$-strongly regular in particular. 
\begin{definition} Let $m$ and $n$ be nonnegative integers with $m+n>1$. A $\Gamma$-semigroup $M$ will be in the class of strongly $(m,n)^{0}-\Gamma$-semigroups, written $M\in s(m,n)^{0}$, if and only if for every $x\in M$ one of the following holds
\begin{enumerate}
	\item $m>0$ and $x\alpha x....\alpha x=0$,
	\item $n>0$ and $x\beta x....\beta x=0$,
	\item $x=x\alpha x....\alpha x\alpha u\beta x\beta x....\beta x$ for some $u\in M$ and $\forall\alpha, \beta\in \Gamma$ where $x^{0}$ is suppressed in the equation when necessary.
\end{enumerate}	
We will say that $M$ is $(m,n)^{0}$-\textit{strongly regular} whenever $M\in s(m,n)^{0}$ and that $M$ is $n^{0}$-\textit{strongly regular} when $M\in s(n,n)^{0}$.
\end{definition}

It is clear that $(m,n)\subseteq (m,n)^{0}$. Indeed, if $M$ is a $\Gamma$-semigroup with no nilpotent elements (other than perhaps 0) we have $M\in s(m,n)$ if and only if $M\in s(m,n)^{0}$.\\

As an extension and generalization of Croisot's (2,2)-regular semigroup, we give the following definition.
\begin{definition} A $\Gamma$-semigroup $M$ with 0 is said to be $2^{0}$-\textit{strongly regular} whenever for each $x\in M$ either $x\gamma x=0$ for some $\gamma\in \Gamma$ or $x\in x\gamma x\gamma M\gamma x\gamma x$ for all $\gamma\in \Gamma$. 
\end{definition}

It is clear that the $2^{0}$-strongly regular $\Gamma$-semigroups are $2^{0}$-regular. It is also clear that the $\Gamma$-semigroups $M$ such that $M\in (1,1)$ are regular $\Gamma$-semigroups. For analogy, we have:
\begin{definition} A left strongly regular $\Gamma$-semigroup is a $\Gamma$-semigroup $M$ such that $M\in s(0,2)$, that is, for all $x\in M$ and $\alpha \in \Gamma$, there exists $u\in M$, such that $x=u\alpha x\alpha x$.
\end{definition}
\begin{definition} A right strongly regular $\Gamma$-semigroup is a $\Gamma$-semigroup $M$ such that $M\in s(2,0)$, that is, for all $x\in M$ and $\alpha \in \Gamma$, there exists $u\in M$ such that $x=x\alpha x\alpha u$.
\end{definition}

Using all above, we fix now all the inaccuracies occured mainly in section 2 \cite{HD}. 

The Lemma 2.13 \cite{HD} has to be stated in the form of Lemma 1.7.

The Lemma 2.14 \cite{HD} has to be stated in the form of Lemma 1.8.

The Theorem 2.15 \cite{HD} has to be stated in the following form:\\
\textbf{Theorem 2.15.} \textit{A $\Gamma$-semigroup $M$ is a union of $\Gamma$-groups if and only if it is at the same time a left and right strongly regular $\Gamma$-semigroup}. \\
\begin{proof} It is clear from the Theorem 1.10.
\end{proof}

The Corollary 2.16 \cite{HD} has to be stated in the following form:\\
\textbf{Corollary 2.16.} \textit{Let $M$ be a $\Gamma$-semigroup. The following statements are equivalent:}
\begin{enumerate}
	\item $M\in s(2,2)$
	\item $M\in s(2,0)\cap s(0,2)$.
	\item $M$ \textit{is a union of $\Gamma$-groups.}
\end{enumerate}
\begin{proof} It follows easily from Theorem 2.15.
\end{proof}

The Proposition 2.17 \cite{HD} has to be stated in the following form:\\
\textbf{Proposition 2.17.} $s(n,n)^{0}=s(0,n)^{0}\cap s(n,0)^{0}$ for $n\geq 2$.\\
\begin{proof} See Theorem 2.15.
\end{proof}

The Definition 2.18 \cite{HD} has to be stated in the following form:\\
\textbf{Definition 2.18} (1). A $\Gamma$-semigroup $M$ is said to be \textit{left-0-strongly regular} if $M\in s(0,2)^{0}$.

(2). A $\Gamma$-semigroup $M$ is said to be \textit{right-0-strongly regular} if $M\in s(2,0)^{0}$.

(3). A $\Gamma$-semigroup $M$ is said to be \textit{intra-0-strongly regular} if for each $x\in M$ either $x\gamma x=0$ for some $\gamma \in \Gamma$ or $x=u\alpha x\alpha x\alpha v$ for some $u, v\in M$ and $\forall \alpha\in \Gamma$.

The Proposition 2.20 \cite{HD} has to be stated in the following form:\\
\textbf{Proposition 2.20} Let $M$ be a $\Gamma$-semigroup with 0. The following statements are equivalent:
\begin{enumerate}
	\item $M$ is [left, right] intra-0-strongly regular.
	\item $M$ is [left, right] 0-semiprime.
	\item If $x\in M$ and $x\gamma x\neq 0, \forall \gamma \in \Gamma$, then $[x\mathcal{L} x\gamma x, x\mathcal{R} x\gamma x]      x\mathcal{J} x\gamma x, \forall \gamma\in \Gamma$.
	\item If $x\in M$ and $x\gamma x\neq 0, \forall \gamma \in \Gamma$, then $[x\in M\Gamma x\gamma x, x\in x\gamma x\Gamma M]      x\in M\Gamma x\gamma x\Gamma M, \forall \gamma \in \Gamma$.
\end{enumerate}

The Proposition 2.21 \cite{HD} has to be stated in the following form:\\
\textbf{Proposition 2.21}. \textit{Let $M$ be a $\Gamma$-semigroup with 0. All left, right and two-sided ideals of $M$ are 0-semiprime if and only if $M$ is $2^{0}$-strongly regular}.

The Theorem 2.22 \cite{HD} has to be stated in the following form:\\
\textbf{Theorem 2.22} \textit{Let $M$ be a $\Gamma$-semigroup with 0. $M$ is $2^{0}$-strongly regular if and only if for all $x\in M$ and $\gamma\in \Gamma$, $x\gamma x=0$ or $x\gamma x\in H_{x}$}.\\
\begin{proof} Assume $M$ is $2^{0}$-strongly regular, that is $M\in s(2,2)$. Let $x\in M$ and $\gamma\in \Gamma$. Let us suppose that $x\gamma x\neq 0$. By Corollary 2.16 and Theorem 2.15 it follows that $M$ is a left and right strongly regular $\Gamma$-semigroup. By Lemmas 1.7 and 1.8 it follows that $x\mathcal{H}(x\gamma x)$, that is $x\gamma x\in H_{x}$.   

Conversely, suppose that for all $x\in M$ and $\gamma\in \Gamma$, $x\gamma x=0$ or $x\gamma x\in H_{x}$. In the former case the first part of Definition 2.2 is satisfied; in the later case by the Theorem 2.2, $H_{x}$ is subgroup of $M_{\gamma}$ and so, the equation $x=x\gamma x\gamma u\gamma x\gamma x$ is the solvable for $u$ in $H_{x}$. Thus in either case $M$ is $2^{0}$-strongly regular.
\end{proof}

The Corollary 2.23 \cite{HD} has to be stated in the following form:\\
\textbf{Corollary 2.23} \textit{Let $M$ be a $2^{0}$-strongly regular $\Gamma$-semigroup. Then all the irregular elements of $M$ lie in $\mathcal{D}$-classes, $D$, such that $D\Gamma D=\{0\}$}.\\
\begin{proof} See the proof in \cite{HD} replacing $2^{0}$-regular with $2^{0}$-strongly regular.
\end{proof}

The Theorem 2.24 \cite{HD} has to be stated in the following form:\\
\textbf{Theorem 2.24} \textit{Let $M$ be a $2^{0}$-strongly regular $\Gamma$-semigroup and suppose $D$ is a nonzero regular $\mathcal{D}$-class union $\{0\}$. Then $D$ is itself a completely 0-simple $\Gamma$-semigroup}.\\
\begin{proof} See the proof in \cite{HD} replacing $2^{0}$-regular with $2^{0}$-strongly regular.
\end{proof}

The Theorem 2.25 \cite{HD} has to be stated in the following form:\\
\textbf{Theorem 2.25} Let $M$ be a regular $\Gamma$-semigroup with 0. Then the following are equivalent:
\begin{enumerate}
	\item $M$ is $2^{0}$-strongly regular.
	\item $M$ is the $0$-disjoint union of sub-$\Gamma$-semigroups which are themselves completely 0-simple $\Gamma$-semigroups.
	\item $M$ is left 0-strongly regular and right 0-strongly regular.
	\item All left and right ideals of $M$ are 0-semiprime.
\end{enumerate}
\begin{proof} See the proof in \cite{HD} replacing $2^{0}$-regular with $2^{0}$-strongly regular.
\end{proof}

The Proposition 3.7 \cite{HD} has to be stated in the following form:\\
\textbf{Proposition 3.7} If $M$ is a $\Gamma$-semigroup with 0 such that each $\mathcal{H}$-class union $\{0\}$ is a quasi-ideal, then $M$ is $2^{0}$-strongly regular.\\
\begin{proof} Let $x\in M$ and let $H=H^{0}_{x}$. Then $H\Gamma H\subseteq M\Gamma H\cap H\Gamma M\subseteq H$ since $H$ is a quasi-ideal. Thus either $x\gamma x=0$ for some $\gamma \in \Gamma$ or $x\gamma x\in H$ for all $\gamma \in \Gamma$. It follows from Theorem 2.22 that $M$ is $2^{0}$-strongly regular.
\end{proof}

The Theorem 4.5 \cite{HD} has to be stated in the following form:\\
\textbf{Theorem 4.5} A $\Gamma$-semigroup $M$ with 0 is absorbent if and only if it is $2^{0}$-strongly regular and the collection of its $\mathcal{D}$-classes union $\{0\}$ is mutually annihilating.\\
\begin{proof} See the proof in \cite{HD} replacing $2^{0}$-regular with $2^{0}$-strongly regular.
\end{proof}

\end{document}